\documentclass[11pt]{article}
\usepackage[utf8]{inputenc}

\usepackage[numbers,sort&compress]{natbib}

\usepackage{geometry}
\usepackage{amsthm,amssymb,amsmath,mathtools}
\usepackage{comment}
\usepackage[spaces,hyphens]{url}

\usepackage{hyperref}
\usepackage{blindtext}
\usepackage{titlesec}
\usepackage{bbm}
\hypersetup{
  colorlinks   = true, 
  urlcolor     = blue, 
  linkcolor    = black, 
  citecolor   = red 
}
\usepackage{pdfpages}
\usepackage[all]{hypcap}
\usepackage[font=small]{caption}
\usepackage[rightcaption]{sidecap}
\usepackage{amsthm}

\theoremstyle{plain}
\newtheorem{theorem}{Theorem}[section]
\newtheorem{lemma}[theorem]{Lemma}
\newtheorem{corollary}[theorem]{Corollary}
\newtheorem{proposition}[theorem]{Proposition}
\theoremstyle{remark}
\newtheorem{remark}[theorem]{Remark}

\theoremstyle{definition}
\newtheorem{definition}[theorem]{Definition}

\numberwithin{equation}{section}

\title{A transform approach to the supercooled Stefan problem}
\author{Thomas Blore\thanks{Mathematical Institute, University of Oxford. \textit{Email:} thomas.blore@maths.ox.ac.uk } \, and Ben Hambly\thanks{Mathematical Institute, University of Oxford. \textit{Email:} ben.hambly@maths.ox.ac.uk }}

\begin{document}
\maketitle
\begin{abstract}
We consider a generalisation of the probabilistic reformulation of the supercooled Stefan problem from \cite{3} allowing for non-integrable initial data. From an application of Ito's formula, we obtain a relationship between the Laplace transform of the initial data and a transform of the solution, together with its possibly discontinuous free boundary. 
Using this relationship we give a sharp existence result for global solutions, we extend previous results on the uniqueness of solutions, and we determine the asymptotic behaviour of solutions for a certain class of initial data.
\end{abstract}

\section{Introduction}
The one-dimensional one-phase supercooled Stefan problem is a model for the freezing of a supercooled liquid that may be formulated as the following free boundary problem for a function $V:\mathbb{R}_+\times\mathbb{R}_+\to \mathbb{R}$ and a free boundary $\Lambda:\mathbb{R}_+\to\mathbb{R}_+$:
\begin{align}
    \partial_t V(t,x) &=\frac{1}{2}\partial_{xx} V(t,x), \quad x>\Lambda_t,\nonumber\\
     \dot{\Lambda}_t &=\frac{1}{2}\partial_x V(t,\Lambda_t),\label{pde}\\
    V(t,\Lambda_t) &=0, \qquad\qquad t\geq 0, &\nonumber\\
    V(0,x) &=g(x),\qquad\;\; x>0.\nonumber
\end{align}

The function $-V(t,x)$ is the temperature at time $t$ and position $x$, where the freezing temperature is taken as zero. The free boundary $\Lambda_t$ is the position of the interface between the solid and liquid phases. We assume that the initial data $g$ is non-negative, so the temperature $-V(0,x)\leq 0$ is below the freezing point, leading to growth of the solid phase. 

For a generic function $g$, classical solutions to \eqref{pde} need not exist, and the PDE may exhibit a blow-up in finite time, in that the derivative of $\Lambda$ ceases to exist, and jump discontinuities in the movement of the free boundary may occur \cite{2}. This causes weak-formulations of the system to be ill-posed. For a large class of initial probability densities, this issue was resolved by Delarue et al. \cite{3} through a reformulation of the system as the solution to a McKean-Vlasov equation, with a ``physical" jump condition enforced at discontinuities.

Further properties of the McKean-Vlasov formulation have been investigated, such as regularity of the free boundary \cite{2,3} and uniqueness of solutions under different conditions \cite{5, mustapha2023wellposednesssupercooledstefanproblem, 3, munoz2026freeboundaryregularitywellposedness}. However, the McKean-Vlasov formulation of \cite{3}  is restricted to initial data $g$ which is integrable. This therefore excludes the classical cases of similarity and travelling wave solutions, and the previously studied behaviour of classical solutions for related initial data \cite{Ricci_Weiqing_1991,4}.

We consider a generalisation of the system studied in \cite{3}, given by the McKean-Vlasov system
\begin{align}
    X^x_t &=x+B_t-\Lambda_t,\nonumber\\
    \tau^x &=\inf\{t\geq 0:X^x_t\leq 0\},\label{alternativetomvr}\\
    \Lambda_t &=\int_0^\infty \mathbb{P}(\tau^x\leq t)g(x)\textrm{d}x.\nonumber
\end{align}
The function $B$ is a standard Brownian motion, and $g$ is a positive, bounded, measurable function, i.e. the Lebesgue density of some positive measure, which may not be finite. 
As in \cite{4}, we also enforce that the free boundary has ``physical" jumps:
\begin{align}
\label{alternativephysical}
 \Delta \Lambda_t=\inf\left\{x>0:\int_0^\infty\mathbb{P}(X_t^y\in[0,x],\tau^y>t)g(y)\textrm{d}y<x\right\}.
\end{align}
We define the corresponding temperature
$$V(t,x):=\int_0^\infty \mathbb{P}(y+B_t\in dx,\tau^y>t)g(y)dy,$$
and similarly 
$$V(t^-,x):=\int_0^\infty \mathbb{P}(y+B_{t}\in dx,\tau^y\geq t)g(y)dy.$$ 

The temperature $V(t,x)$ and corresponding free boundary $\Lambda_t$ form a weak solution to the supercooled Stefan problem \cite{2}. For integrable initial data $g$, it has been further determined that at any time $t$, the temperature $V(t+s,x)$ and free boundary $\Lambda_{t+s}$ form a local-in-time classical solution to the supercooled Stefan problem \cite[Theorem 1.1]{3}. Therefore, the McKean-Vlasov system \eqref{alternativetomvr} can be used to study the supercooled Stefan problem.
We shall refer to $g,$ $V(t^-,.)$ and $V(t,.)$ as densities throughout this work since each corresponds to the density of some positive measure with respect to the Lebesgue measure on $\mathbb{R}^+$. 

We will take a transform approach to the supercooled Stefan problem. 
Let $\phi(x)=e^{-\lambda x}$ and apply Ito's formula to terms of the form 
\[\phi(X_{0^-}+B_t-\Lambda_t)\mathbbm{1}_{\{\inf_{s\leq t}(X_{0^-}+B_s-\Lambda_s)>0\}}. \] 
This gives an expression for the Laplace transform of the temperature $V$, without requiring the existence of a classical solution. 

\begin{proposition}[Laplace transform of solutions]
\label{Ito}
For all $\lambda>0$, any solution to \eqref{alternativetomvr} on the interval $[0,t]$ satisfies the following equation:
\begin{align}&\int_{\Lambda_t}^\infty V(t,x)e^{-\lambda x-\frac{\lambda^2t}{2}}\textrm{d}x+\int_0^\infty (1-g(x))e^{-\lambda x}\textrm{d}x-\frac{1}{\lambda}e^{-\lambda \Lambda_t-\frac{\lambda^2t}{2}}\\&=\frac{\lambda}{2}\int_0^te^{-\lambda\Lambda_{s^-}-\frac{\lambda^2s}{2}}\textrm{d}s+\sum_{s\leq t}e^{-\frac{\lambda^2s}{2}}\left(-\int_{\Lambda_{s^-}}^{\Lambda_s} V(s^-,x)e^{-\lambda x}\textrm{d}x-\frac{1}{\lambda}\Delta e^{-\lambda\Lambda_s}\right).\nonumber
\end{align}
\end{proposition}

We will use simple manipulations of this expression to obtain several new results on the existence, uniqueness, and long-time behaviour of physical solutions to \eqref{alternativetomvr}. 
In Section \ref{altcond} we give a sharp condition for the existence of global-in-time solutions to \eqref{alternativetomvr}, with finitely many jumps. In Section~\ref{uniquenessappendix}, Corollary~\ref{contunique2}, we give a global uniqueness result that extends the work of \cite{mustapha2023wellposednesssupercooledstefanproblem} and does not require the condition that the monotonicity of the initial data changes finitely often on compact sets, as considered in \cite{3}. Our result is also an extension to the case of non-integrable initial data of \cite[Theorem 1.7]{munoz2026freeboundaryregularitywellposedness}, using different arguments. 
 Finally, in Section \ref{asymptoticsection} we give conditions under which $\Lambda_t\sim Ct^\alpha$ for $1/2\leq \alpha\leq 1$, providing a rigorous proof of the work in \cite{4}, and allowing for the presence of discontinuities. 

\begin{remark}
    The system \eqref{alternativetomvr} can be rewritten by replacing the third equation with
$$\Lambda_t =\int_0^\infty \mathbb{P}(\tau^x\leq t)G(\textrm{d}x)$$
for a measure $G$, which need not be absolutely continuous with respect to Lebesgue measure. The results can then be rewritten in terms of properties of $G$. For simplicity we shall only work with $G$ which possesses a bounded Lebesgue density $g$. 
\end{remark}

\section{Global existence of solutions}
\label{altcond}

In this section, we give a sharp condition for the existence of a global-in-time solution with finitely many jumps.

\begin{itemize}
\item Condition (A): 
    $$\limsup_{\lambda\rightarrow 0}\frac{1}{\lambda}\int_0^\infty (1-g(x))e^{-\lambda x}\textrm{d}x=\infty .$$
\end{itemize}
Note that the function $g$ need not be bounded above by one. Thus, the integrand can be negative, and the integral may be negative for some values of $\lambda$.
\begin{theorem}
\label{conda4}
    Suppose that $g$ is a positive, bounded, measurable function that satisfies condition (A). Then a global, physical solution to \eqref{alternativetomvr} with initial density $g$ exists.

    Conversely, suppose that there exists a solution to \eqref{alternativetomvr} with initial data $g$ which is a positive, bounded, measurable function. If at least one solution
    possesses finitely many discontinuities, then the initial density $g$, must satisfy Condition \hyperlink{A}{(A)}.
\end{theorem}

In order to show this, we first require the notion of a minimal solution. As in the work \cite{minimal}, we shall call a function $f$ a minimal solution of \eqref{alternativetomvr} if, for any solution $\Lambda$ of \eqref{alternativetomvr}, $f_t\leq\Lambda_t$ for all $t\geq 0$. We shall write $\underline{\Lambda}$ to denote a minimal solution to \eqref{alternativetomvr} with initial data $g$ throughout this paper. This minimal solution to the system \eqref{alternativetomvr} can be seen to satisfy the ``physical" jump condition \eqref{alternativephysical} via an identical argument to that of \cite[Theorem 4.7]{ledger2024supercooledstefanproblemtransport}. Furthermore, given the existence of some solution to \eqref{alternativetomvr}, the existence of a minimal solution must follow. To show this, we first note, as in \cite{DIRT1, 2, NS19, minimal}, that solutions to \eqref{alternativetomvr} are equivalent to fixed points of the following operator $\Gamma:D([0,\infty),\mathbb{R})\rightarrow D([0,\infty),\mathbb{R})$:
\begin{align}
    &X_t^{x}(f)=x+B_t-f_t,\nonumber\\
    &\tau^x(f)=\inf\{t\geq 0:X_t^{x}(f)\leq 0\},\label{gammadefn}\\
    &\Gamma(f)_t:=\int_0^\infty \mathbb{P}(\tau^x(f)\leq t)g(x)dx.\nonumber
\end{align}
It is easily seen that $\Gamma$ is a monotone operator on the space of cadlag functions, and that if a fixed point $\Lambda$ exists, $\Gamma(0)\leq \Gamma(\Lambda)=\Lambda$. The existence of a minimal solution then holds by applying Tarski's Theorem as in  \cite[Theorem 4.1]{ledger2024supercooledstefanproblemtransport}. 

To prove the result of Theorem \ref{conda4}, we then consider the value of solutions to \eqref{alternativetomvr} replacing the initial density $g(x)$ by $g(x)\mathbbm{1}_{[0,R]}(x)$. Using Condition \hyperlink{A}{(A)} in the formula of Proposition \ref{Ito}, we obtain a uniform bound on solutions for such an initial density. Taking a limit over $R$, we deduce the existence of a minimal solution to \eqref{alternativetomvr}. 

\begin{proof}

For $R>0$ we define the density $g_R$ by $g_R(x):=g(x)\mathbbm{1}_{[0,R]}(x)$ for $x\geq 0$. A minimal solution to \eqref{alternativetomvr} with this initial condition can be shown to exist via an identical argument to \cite[Theorem 4.1]{ledger2024supercooledstefanproblemtransport}. Let $\Lambda^R$ be the minimal solution to \eqref{alternativetomvr} with initial density $g_R$. Consider an arbitrary $T>0$ and suppose that there exists a constant $C$ such that $$\sup_{R\geq 0}\Lambda_T^{R}\leq  C.$$ 

We define $\Lambda^R(T)\in D([-1,\infty),\mathbb{R})$ via $\Lambda^R(T)_t=\Lambda^R_{t\wedge T}\mathbbm{1}_{\{t\geq 0\}}$. Since we assumed that $\sup_{R\geq 0}||\Lambda^{R}(T)||_\infty\leq C,$ by the arguments of \cite[Theorem 12.12.2]{whitt} the set of functions $\{\Lambda^{R}(T) \textrm{ for }{R\geq 0}\}$ is a pre-compact set under the M1 Skorokhod topology. Therefore, for any sequence $R_m$ such that $R_m\uparrow\infty$ as $m\rightarrow\infty$, there is a further subsequence $R_{m_k}$ such that $\Lambda^{R_{m_k}}(T)$ converges under the Skorokhod M1 topology as $k\rightarrow\infty$. From a slight modification to the arguments of \cite[Proposition 2.1]{minimal}, we can see that $\Lambda^{R_{m_k}}_t(T)$
converges under the M1 Skorokhod topology to some process $\Lambda_t(T)$ that satisfies
\begin{align}\label{a4sol}\Lambda_t(T)=\int_0^\infty \mathbb{P}(\inf_{s\leq t}(x+B_s-\Lambda_s(T))\leq 0) g(x)\textrm{d}x, \hspace{0.4cm} t\leq T.\end{align}
Such a process $\Lambda_t(T)$ is a solution to \eqref{alternativetomvr} on $[0,T]$ and therefore a minimal solution exists on $[0,T]$ from identical arguments to those used for the existence on $[0,\infty)$. 

Therefore, if we can determine that for all $T>0$ there exists a constant $ C$ such that $\sup_{R\geq 0}\Lambda_T^{R}\leq C$, we will obtain the existence of a global-in-time minimal solution, and therefore the existence of a global-in-time physical solution.

To determine the existence of a bound on $\Lambda_T^{R}$ over $R$, we suppose, for a contradiction, that for some fixed $T\geq 0$, there is some sequence $R_k$ such that $\Lambda_T^{ R_k}\rightarrow\infty$. From Proposition \ref{Ito}, taking the limit as $t\rightarrow\infty$, it follows that
\begin{align*}
&\int_0^\infty (1-g_{R_k}(x))e^{-\lambda x}\textrm{d}x\\
&=\frac{\lambda}{2}\int_0^\infty e^{-\lambda \Lambda^{R_k}_{t^-}-\frac{\lambda^2t}{2}}dt+\sum_{\Delta \Lambda^{R_k}_t>0}e^{-\frac{\lambda^2t}{2}}\left(\int_{\Lambda^{R_k}_{t^-}}^{\Lambda^{R_k}_t}\left(-e^{-\lambda x}\right) V(t^-,x)\textrm{d}x-\frac{1}{\lambda}\Delta e^{-\lambda \Lambda^{R_k}_t}\right).\end{align*}
In this expression, the density $V(t^-,.)$ corresponds to the previous definition, using the initial data $g_{R_k}$.

By the physical jump condition, $\int_{\Lambda^{R_k}_{t^-}}^{\Lambda^{R_k}_{t^-}+x} V(t^-,y)\textrm{d}y\geq x$ for $x\leq \Delta \Lambda^{R_k}_t$. Therefore, since $e^{-\lambda x}$ is decreasing, each term in the sum is bounded above by $$-\int_{\Lambda^{R_k}_{t^-}}^{\Lambda^{R_k}_t} e^{-\lambda x}\textrm{d}x-\frac{1}{\lambda}\Delta e^{-\lambda \Lambda^{R_k}_t}=0.$$
It follows that
\begin{align} \label{eq:gbd}
    &\int_0^\infty (1-g_{R_k}(x))e^{-\lambda x}\textrm{d}x
    \leq \frac{\lambda}{2}\int_0^\infty e^{-\lambda \Lambda^{R_k}_{t^-}-\frac{\lambda^2t}{2}}dt.
\end{align}
        
For an arbitrary $\lambda>0$, consider the process $h(\lambda,R_k)_t=e^{-\lambda \Lambda^{R_k}_t}$. This is bounded between 0 and 1, and monotone decreasing. Using identical extension arguments as for $\Lambda^R(T)$, we obtain extensions of each of these processes to $D([-1,\infty),\mathbb{R})$, so that the set of such processes is pre-compact under the M1 Skorokhod topology. We may therefore obtain a subsequence of $h(\lambda,R_k)_t$ which converges Lebesgue almost everywhere to a cadlag function $h(\lambda)_t$. This function is bounded above by $1$ and, by the assumption that $\Lambda_T^{ R_k}\rightarrow\infty$, we have $h(\Lambda)_t=0$ for $t\geq T$.

Applying the Dominated Convergence Theorem in \eqref{eq:gbd} and using that $h(\lambda)_t\leq 1$, it holds that
$$ \int_0^\infty ( 1-g(x))e^{-\lambda x}\textrm{d}x\leq \frac{\lambda}{2} \int_0^T h(\lambda)_te^{-\frac{\lambda^2t}{2}}dt \leq \frac{\left(1-e^{-\frac{\lambda^2T}{2}}\right)}{\lambda}.$$
Since $\lambda>0$ was arbitrary, dividing both sides by $\lambda>0$, and taking the limsup over $\lambda \downarrow 0$, we have $\frac{T}{2}\geq \infty$, a contradiction. Therefore, the assumption that $\Lambda_T^{R_k}\rightarrow\infty$ cannot hold, and we have proved the first part of the theorem.

To prove the second part of the theorem, we again consider the expression given in Proposition \ref{Ito}. Bounding the sum term above, we obtain
\begin{align*}
    &\frac{1}{\lambda }\int_0^\infty (1-g(x))e^{-\lambda x}\textrm{d}x\geq \frac{1}{2}\int_0^\infty e^{-\lambda \Lambda_{t}-\frac{\lambda^2t}{2}}dt-\sum_{\Delta\Lambda_t>0}(\Delta\Lambda_t)^2-{}O_{\lambda \downarrow 0}(\lambda).
\end{align*}
Applying Fatou's lemma, it is simple to see that the right hand side grows to infinity as $\lambda\downarrow 0$, whenever $\sum_{\Delta\Lambda_t>0}(\Delta\Lambda_t)^2$ is finite.
\end{proof}

\section{Uniqueness}

\label{uniquenessappendix}
In this section, we give some conditions for the global uniqueness of physical solutions to \eqref{alternativetomvr}. We assume throughout that a global-in-time minimal solution to \eqref{alternativetomvr} exists.

\begin{theorem}
\label{contunique}
    Suppose that there exists $\epsilon>0$ such that the minimal solution $\underline{\Lambda}$ is continuous on $(0,\epsilon)$. Then the physical solution to \eqref{alternativetomvr} is unique.
\end{theorem}

We shall first give several corollaries to this result before giving the proof.

\begin{corollary}
\label{contunique2}
    Suppose that $\int_{0}^xg(y)\textrm{d}y\leq  x$ for all $x>0$, and that a global-in-time minimal solution to \eqref{alternativetomvr} exists. Then the physical solution to \eqref{alternativetomvr} is continuous on $(0,\infty)$ and is unique.
\end{corollary}

The condition given in the statement of this proposition is nearly that of \cite[Theorem 1.1]{PDEresult} under which classical solutions to the supercooled Stefan problem exist. Using approximations of the density, it is simple to check that the minimal solution to \eqref{alternativetomvr} is continuous on $(0,\infty)$. Details are given at the end of this section.

The uniqueness results of \cite{3} extend immediately to allow for $g\notin L^1(\mathbb{R}^+)$.
\begin{corollary}
    Let $x^*=\inf\{x>0:\int_0^xg(y)<y\}$. Suppose that $g$ changes monotonicity finitely often on $[x^*,x^*+\epsilon]$ for some $\epsilon>0$. Then the physical solution to \eqref{alternativetomvr} is unique.
\end{corollary}
Uniqueness on a small time interval can be established exactly as in \cite[Proposition 5.2]{3}, after which the arguments of the proof of Corollary \ref{contunique} may be applied.

We finally give a condition which requires only regularity of the data near zero, and without monotonicity change constraints.
\begin{corollary}
\label{contunique3}
Suppose that there exists $\epsilon>0$ such that $g(x)\leq 1$ for $x<\epsilon$ and $\inf\{x:\int_0^xg(y)\textrm{d}y<x\}=0$. Then the physical solution is unique.
\end{corollary}

\begin{proof}
By using the bound $\mathbb{P}(y+B_t-\Lambda_t\leq x,\tau^y>t)\leq \mathbb{P}(\Lambda_t<y+B_t\leq \Lambda_t+x)$ and integrating against the density $g$, it is straightforward to check that for all sufficiently small $t$, $\inf\{x:\int_0^\infty \mathbb{P}(X_t^y\leq x,\tau^y>t)g(y)\textrm{d}y<x\}=0$. Therefore, by the physical jump condition \eqref{alternativephysical}, we determine that any solution to \eqref{alternativetomvr} is continuous on some small time interval. Hence we can conclude the uniqueness by Theorem \ref{contunique}.
\end{proof}

\begin{remark}
The result of Corollary \ref{contunique2} implies that of \cite[Theorem 1.1]{mustapha2023wellposednesssupercooledstefanproblem}. The condition of Corollary \ref{contunique2} also holds for the data studied in \cite[Sections 3,4]{mustapha2023wellposednesssupercooledstefanproblem} where \cite[Theorem 1.1]{mustapha2023wellposednesssupercooledstefanproblem} does not apply.
\end{remark}

To prove Theorem \ref{contunique}, we first use the Laplace transform result of Proposition \ref{Ito} to determine uniqueness on small time intervals. 
\begin{proposition}
\label{uniquecont}
    Suppose that $g$ is bounded, measurable, and that a minimal solution $\underline{\Lambda}$, to \eqref{alternativetomvr} exists on some time interval $[0,T)$. Suppose further that at any time $T'$ where $\underline{\Lambda}$ is discontinuous, there exists a positive $\epsilon_{T'}>0$ such that the minimal solution is continuous on $(T',T'+\epsilon_{T'})$. Then, the physical solution to \eqref{alternativetomvr} is unique.  
\end{proposition}
\begin{proof}
Let $\Lambda_t$ be a physical solution to \eqref{alternativetomvr} on $[0,T)$, which shall be compared to the minimal solution $\underline{\Lambda}_t$. Suppose that $t^*:=\inf\{t\geq 0:\underline{\Lambda}_{t}\neq \Lambda_{t}\}\wedge T$ is strictly less than $T$. Without loss of generality, we may consider $t^*=0$. We can do this by considering the system with the initial density $\tilde V(t^{*-},x)$, which is the density corresponding to $\int_0^\infty \mathbb{P}(y+B_{t^*}-\underline{\Lambda}_{t^*-}\in \textrm{d}x,\tau\geq t^*)g(y)\textrm{d}y$. Furthermore, since the physicality condition \eqref{alternativephysical} uniquely determines the jump size of any physical solution at time zero, we may assume that $\underline{\Lambda}_0=\Lambda_0=0$ by restarting the system with initial density $\tilde V(t^*,x)$. 

By assumption, there exists $\epsilon>0$ such that $\underline{\Lambda}$ is continuous on $[0,\epsilon)$. Since $t^*=0$, and $\Lambda,$ $\underline{\Lambda}$ are right continuous, we can find $\epsilon'>0$, and $0<t_1<t_2$ such that $\Lambda_{t}>\underline{\Lambda_t}+\epsilon'$ for $t\in(t_1,t_2)$.

Let us write $\underline{V}(t,x)$ for the density corresponding to $\int_0^\infty \mathbb{P}(y+B_t\in \textrm{d}x,\inf_{s\leq t}y+B_s-\underline{\Lambda}_s>0)g(y)\textrm{d}y$, and similarly $V(t,x)$ for the same density with $\underline{\Lambda}$ replaced by $\Lambda$.
We consider the expression given in Proposition \ref{Ito} for $\Lambda_t$ and $\underline{\Lambda}_t$, on the interval $[0,t_2]$.
Taking the difference of the expressions, then multiplying both sides by $e^{\lambda \underline{\Lambda}_{t_2}+\frac{\lambda^2t_2}{2}}$ we obtain:
\begin{align*}
    &\int_{\underline{\Lambda}_{t_2}}^\infty \underline{V}(t_2,x)e^{-\lambda x+\lambda \underline{\Lambda}_{t_2}}\textrm{d}x-\int_{\Lambda_{t_2}}^\infty V(t_2,x)e^{-\lambda x+\lambda\underline{\Lambda}_{t_2}}\textrm{d}x-\frac{1-e^{-\lambda(\Lambda_{t_2}-\underline{\Lambda}_{t_2})}}{\lambda}\\
    &=\frac{\lambda}{2}\int_0^{t_2}(e^{-\lambda \underline{\Lambda}_s}-e^{-\lambda \Lambda_{s}})e^{\lambda \underline{\Lambda}_{t_2}+\frac{\lambda^2(t_2-s)}{2}}\textrm{d}s\\
    &\hspace{0.4cm}-e^{\underline{\lambda\Lambda}_{t_2}+\frac{\lambda^2t_2}{2}}\sum_{s\leq t_2}e^{-\frac{\lambda^2s}{2}}\left(-\int_{\Lambda_{s^-}}^{\Lambda_s} V(s^-,x)e^{-\lambda x}\textrm{d}x-\frac{1}{\lambda}\Delta e^{-\lambda\Lambda_s}\right).
\end{align*}

We claim that the left hand side converges to zero as $\lambda\rightarrow\infty$. From an identical argument to that of \cite[Lemma 2.1]{2} we obtain the bound $||g||_\infty\geq V(t_2,x),\underline{V}(t_2,x)$. Therefore, the first two terms on the left are bounded above by $\frac{||g||_\infty}{\lambda}$. The final term is bounded above by $\frac{1}{\lambda}$, and the validity of the claim is clear.

As stated in the proof of Theorem \ref{conda4}, each term in the sum is negative. By definition of the minimal solution $\underline{\Lambda}_t\leq \Lambda_t$ for all $t$, the integral term is non-negative. Therefore, each term on the right hand side of the equation is non-negative. It must therefore be the case that both the sum term and the integral term converge to zero as $\lambda \rightarrow \infty$.

Consider only the integral term. Applying Fatou's lemma, we have
$$0=\int_0^{t_2}\liminf_{\lambda\uparrow\infty}\frac{\lambda}{2}(e^{-\lambda \underline{\Lambda}_s}-e^{-\lambda \Lambda_{s}})e^{\lambda \underline{\Lambda}_{t_2}+\frac{\lambda^2(t_2-s)}{2}}\textrm{d}s.$$
Therefore:
$$\liminf_{\lambda\uparrow\infty}\frac{\lambda}{2}(e^{-\lambda \underline{\Lambda}_s}-e^{-\lambda \Lambda_{s}})e^{\lambda \underline{\Lambda}_{t_2}+\frac{\lambda^2(t_2-s)}{2}}=0 \textrm{ for }a.e. \hspace{0.1cm}s\in[0,t_2].$$ However, for $s\in (t_1,t_2)$, we can bound this by 
\begin{align*}
    \liminf_{\lambda \uparrow\infty }\frac{\lambda}{2}(1-e^{-\lambda \epsilon'})e^{\lambda (\underline{\Lambda}_{t_2}-\underline{\Lambda}_s)+\frac{\lambda^2(t_2-s)}{2}}\geq \liminf_{\lambda\uparrow\infty}\frac{\lambda}{2}(1-e^{-\lambda \epsilon'})e^{\frac{\lambda^2(t_2-s)}{2}}=\infty.
\end{align*}
This is a contradiction, and therefore we must have $\underline{\Lambda}=\Lambda$. 
\end{proof}

To extend this to all times, we note that, by a simple extension to the arguments of \cite[Proposition 3.2]{5}, the result therein extends to the case of $g\notin L^1(\mathbb{R}^+)$ also considered here. 
\begin{proposition}
    For any $t>0$, the function $x\rightarrow V(t^-,x+\Lambda_t)\in C^\infty((0,\infty))$.
    \label{analyticity}
\end{proposition}

We can now prove Theorem \ref{contunique}.
\begin{proof}[Proof of Theorem \ref{contunique}]
By Proposition \ref{uniquecont} the physical solution is unique on $(0,\epsilon/2]$.
At any time of discontinuity of $\underline{\Lambda}_t$ which is greater than $\epsilon/2$, since $\Delta\Lambda_t>0$, by Proposition \ref{analyticity}, $V(t^-,\Lambda_t+x)\in C^\infty((0,\infty))$. By definition of the physical jump condition, from arguments identical to those of \cite[Theorem 3.3]{5}, it therefore must hold that $V(t^-,\Lambda_t+x)\leq 1$ and $\int_0^x1-V(t^-,\Lambda_t+y)\textrm{d}y>0$ in some neighbourhood of zero. From this we readily determine that $\underline{\Lambda}_t$ is continuous on $[t,t+\epsilon'_t)$, for some small $\epsilon'_t$.
\end{proof}
\begin{remark}
    In \cite[Theorem 1.7]{munoz2026freeboundaryregularitywellposedness} it is shown that when $g$ is integrable, physical solutions are globally unique if they are unique on some small time interval $[0,\epsilon)$. From the arguments of the proof of Theorem \ref{contunique}, it is clear that this result extends to the case of non-integrable $g$. Theorem \ref{contunique} can then be seen as a direct improvement of \cite[Theorem 1.7]{munoz2026freeboundaryregularitywellposedness} as it allows for non-integrable initial data, and gives a sufficient condition for the local uniqueness of solutions.
    The main arguments of \cite{munoz2026freeboundaryregularitywellposedness} are based on determining the regularity of solutions, so the results of \cite{3} can be applied. Such regularity would be interesting to investigate in the case of non-integrable initial data, though it is not clear if the Laplace transform results of this work would be useful.
\end{remark}

\begin{proof}[Proof of Corollary \ref{contunique2}]
 We approximate the density $g$ to allow application of the results of \cite{PDEresult}. Let $g_n(x)=g(x)\mathbbm{1}_{[\frac{1}{n},n]}(x)$.
 It is simple to construct a sequence of piecewise continuous, non-negative functions $g_n^{k}$ which are supported on $[\frac{1}{n},n]$, and such that $g_n^k(x)\leq g_n(x)$ and $\int_{\frac{1}{n}}^n|g_n(x)- g_n^k(x)|\textrm{d}x<\frac{1}{k}$ for all $k\in\mathbb{N}$. 

Since $\limsup_{x\downarrow 0}g_n^{k}(x)=0$, and $\int_0^xg_n^{k}(y)\textrm{d}y<x$ for all $x>0$, the densities $g_n^k$ satisfy the conditions of 
\cite[Theorem 1.1]{PDEresult}. Therefore, there is a global-in-time classical solution to the supercooled Stefan problem \eqref{pde} with initial data $g_n^{k}$. We denote this solution by $\Lambda^{n,k}_t$. It is simple to check that ${\Lambda}^{n,k}_t$ is the unique solution to \eqref{alternativetomvr} with initial density $g_n^{k}$. 
It is also simple to observe that since $g_n^k\leq g$, $\Lambda_t^{n,k}\leq \underline{\Lambda}_t$ for all $t\geq 0$.

From a simple adaptation to the arguments of \cite[Proposition 2.1]{minimal}, it may be seen that as $n\rightarrow\infty$, $\Lambda^{n,k}$ converges Lebesgue almost everywhere to the minimal solution to \eqref{alternativetomvr} with initial density $g_n$, labelled as $\Lambda^n$. Then, as $n\rightarrow\infty$, $\Lambda^n$ converges Lebesgue almost everywhere to $\underline{\Lambda}$, i.e:

$$\lim_{n\rightarrow\infty }\lim_{k\rightarrow\infty}{\Lambda}_t^{n,k}=\underline{\Lambda}_t \textrm{ for Lebesgue }a.e. \hspace{0.15cm} t>0.$$

Applying the formula of Proposition \ref{Ito}, we determine that for any $n,k\in \mathbb{N}$, 
\begin{align}\label{ito3}\frac{1}{\lambda}-\int_0^\infty g_n^k(x)e^{-\lambda x}\textrm{d}x=\frac{\lambda}{2}\int_0^\infty e^{-\lambda \Lambda^{n,k}_s-\frac{\lambda^2s}{2}}\textrm{d}s.\end{align}

Applying the Dominated Convergence Theorem, we can take the limit over $k$, then over $n$ in the above expression to obtain that

$$\frac{1}{\lambda}-\int_0^\infty g(x)e^{-\lambda x}\textrm{d}x=\frac{\lambda}{2}\int_0^\infty e^{-\lambda \underline{\Lambda}_s-\frac{\lambda^2s}{2}}\textrm{d}s.$$

Comparing this to the expression in Proposition \ref{Ito}, we see can see that at any jump time it must hold that $V(s^-,x)=1$ for almost every $x\in(\Lambda_{s^-},\Lambda_s)$. Recall from  Proposition \ref{analyticity} that $V(s^-,.)$ is smooth for $s>0$. Therefore at any jump time $s>0$, $V(s^-,.)$ must be one almost everywhere, which would mean that $\Lambda$ must jump to infinity. Since this can not occur for a global in time solution, there can be no jumps after time zero, hence $\Lambda$ is continuous on $(0,\infty)$. 

Uniqueness then follows immediately from  Proposition \ref{uniquecont}.
\end{proof}

\begin{remark}
 Non-uniqueness of physical solutions to \eqref{alternativetomvr} on a small time interval can only occur if there is a sequence of times $t_n\downarrow 0$ such that $\Delta \underline{\Lambda}_{t_n}>0.$ This must occur for any initial density $g$ such that for all $t>0$, there exists some $\lambda>0$ such that $ \int_0^\infty (1-g(x))e^{-\lambda x}\textrm{d}x<-\frac{||g||_\infty e^{-\frac{\lambda^2t}{2}}}{\lambda}$, and such that $\inf\{x:\int_0^x(1-g(y))dy>0\}=0$. In this case no jump can occur at time zero. If the first jump occurred after a time $\delta$ we could consider the expression given in Proposition \ref{Ito}, with $t=\delta/2$, and note that $\int_0^\infty (1-g(x))e^{-\lambda x}\textrm{d}x$ is greater than a term of size $\frac{-||g||_\infty e^{-\frac{\lambda^2 \delta}{4}}}{\lambda}$, plus terms which are positive.
\end{remark}

\section{Asymptotic speed}
\label{asymptoticsection}

In this section, we show that for certain classes of initial data $g$, the asymptotic motion of $\Lambda_t$ may be determined. 
\begin{proposition}
\label{nontravelling}
    Let $C>0$ and $1/2\leq \alpha\leq 1$. Suppose that $g$ is bounded and that 
    $$\lim_{x\rightarrow\infty}\frac{\int_0^x(1-g(y))dy}{ x^{1/\alpha-1}}=C.$$
    Then the asymptotic speed is given by $$\lim_{t\rightarrow\infty }\frac{\Lambda_t}{t^{\alpha}}=\left(\frac{(1-\alpha)}{2\alpha^2C}\right)^{\alpha}.$$
\end{proposition}

This result was conjectured in \cite[Section 4]{4}, for smooth solutions to the supercooled Stefan problem, and shown to hold for $\alpha=1/2,1$ when classical solutions exist in \cite{Ricci_Weiqing_1991}. In order to show this result, we compare $g$ to initial data $g_{C,\alpha}$ for which $\Lambda_t=Ct^\alpha$. We therefore first show that such data $g_{C,\alpha}$ exists. 

\begin{lemma}
    For any $C>0$, $1/2\leq \alpha\leq 1$, there exists a positive function $g_{C,\alpha}$ such that the solution to the supercooled Stefan problem \eqref{pde} with initial data $g_{C,\alpha}$ is given by $\Lambda_t=Ct^{\alpha}$.
\end{lemma}

\label{asymptoticappend}
\begin{proof}
We shall consider $C=1$ to reduce the notation. The proof is identical to that for $C\neq 1$. 
We shall consider $1/2<\alpha<1$ since the existence of similarity solutions ($\alpha=1/2$) and travelling wave solutions ($\alpha=1$) are well known. We shall construct the solution using the formula given by Proposition \ref{Ito}.

We first recall the definitions of a completely monotone function and a Bernstein function.
\begin{definition}
    A function $f$ is called completely monotone if it is continuous on $[0,\infty),$ infinitely differentiable on $(0,\infty)$, and satisfies $(-1)^n\frac{\textrm{d}^nf}{\textrm{d}x^n}(x)\geq 0$ for all non-negative integers $n$.
\end{definition}
\begin{definition}
    A Bernstein function is a non-negative differentiable function, with derivative which is completely monotone.
\end{definition}

We define a function $F:\mathbb{R}^+\times \mathbb{R^+}$ by
\begin{align*}
F(\lambda,t)&=e^{\lambda t^\alpha+\lambda^2t/2}\int_t^\infty e^{-\lambda s^\alpha-\lambda^2s/2}\alpha s^{\alpha-1}\textrm{d}s\\
    &=\frac{1}{\lambda}\int_0^\infty e^{-u}e^{-\frac{\lambda^2}{2}((\frac{u}{\lambda}+t^\alpha)^{1/\alpha}-t)} \textrm{d}u \textrm{ taking }u=\lambda(s^\alpha-t^\alpha).
\end{align*}

We claim that for each $t\geq 0$, $F(\lambda,t)$ is a completely monotone function of $\lambda$. To show this, it is sufficient to check that for all $u>0$, $t\geq 0$, \\$m( \lambda)=\frac{\lambda^2}{2}((\frac{u}{\lambda}+t^\alpha)^{1/\alpha}-t)$ is a Bernstein function in $\lambda$. It would then follow that
\\$\exp\left(-\frac{\lambda^2}{2}((\frac{u}{\lambda}+t^\alpha)^{1/\alpha}-t)\right)$ is completely monotone. Therefore the integrand in the above integral will also be completely monotone, and thus so is $F(.,t)$. The first derivative of $m(\lambda)$ can be explicitly computed, as can its inverse Laplace transform, from which it is easily seen that $m'(\lambda)$ is the Laplace transform of a positive function and hence a completely monotone function. It is then simple to see that $m(\lambda)$ is a Bernstein function.

We have shown that for all $t\geq 0$, $F(\lambda,t)$ is completely monotone in $\lambda$. Therefore, by Bernstein's theorem, for all $t\geq 0$, there exists a positive function $V_t:\mathbb{R}^+\rightarrow \mathbb{R^+}$ such that $V_t$ has the Laplace transform $F(\lambda,t)$. 

We define $g=V_0$. We write $(f,g)$ for the integral $\int_0^\infty f(x)g(x)dx$.  
By definition, the functions $V_t$ are such that for all $\lambda\geq 0$, $t\geq 0$:
\begin{align*}
    (V_t,e^{-\lambda x})&=(g,e^{-\lambda x})+\int_0^t\frac{1}{2}(V_s,\partial_{xx}\left(e^{-\lambda x}\right))\textrm{d}s\\&-\int_0^t\alpha s^{\alpha-1}(V_s,\partial_x\left(e^{-\lambda x}\right))\textrm{d}s-\int_0^t\alpha s^{\alpha -1}(e^{-\lambda x}|_{x=0})\textrm{d}s.
\end{align*}

By a standard density argument we extend this to 
\begin{align*}
    (V_t,\phi)=(g,\phi)+\int_0^t\frac{1}{2}(V_s,\phi'')\textrm{d}s-\int_0^t\alpha s^{\alpha-1}(V_s,\partial_x\phi)\textrm{d}s-\int_0^t\alpha s^{\alpha -1}\phi(0)\textrm{d}s
\end{align*}
for any $t\geq 0$ and $\phi$ a Schwartz function. Therefore, the function $V(t,x)=V_t(x+t^\alpha)$ and the free boundary $t^\alpha$, form a weak solution to the supercooled Stefan problem \eqref{pde} with initial data $g$. Since the function $t^\alpha$ is differentiable on $(0,\infty)$, it is straightforward to argue that $V(t,x)$ is a classical solution to the supercooled Stefan problem for $t>0$.
\end{proof}

\begin{proof}[Proof of Proposition \ref{nontravelling}]

To simplify our expressions we check the equivalent statement that if $\lim_{x\rightarrow\infty}\frac{\int_0^x(1-g(y))dy}{x^{1/\alpha-1}}=\frac{(1-\alpha)}{2\alpha^2C^{1/\alpha}},$ then $\lim_{t\rightarrow\infty}\Lambda_t/t^\alpha=C$. 
We first note that by definition of $g_{C,\alpha}$, it is simple to check that for any $C>0,$ $1/2\leq \alpha\leq 1$, $$\lim_{x\rightarrow\infty}\frac{\int_0^x(1-g_{C,\alpha}(y))dy}{ x^{1/\alpha-1}}=\frac{(1-\alpha)}{2\alpha^2C^{1/\alpha}}.$$

Fix an arbitrary $\epsilon>0$. Combining the above result with the assumption on $g$, it is simple to observe that there exists $d_\epsilon$ such that for all $x>0$, $\int_0^xg(y)dy$ is bounded above by 

$$ F(x)=(||g||_\infty x)\wedge d_\epsilon+\int_0^{(x- d_\epsilon)_+}g_{C+\epsilon,\alpha}(y)dy.$$ 
The physical solution to \eqref{alternativetomvr} with initial data $\frac{dF}{dx}$ is an upper bound for the physical solution with initial data $g$. This physical solution is given exactly by $||g||_\infty d_\epsilon+(C+\epsilon)t^{\alpha}$.

    Since $\epsilon>0$ is arbitrary, we have therefore determined that $$\limsup_{t\rightarrow\infty}\frac{\Lambda_t}{t^{\alpha}}\leq C.$$

    Consider the set of functions $f^x(t)=x\Lambda_{t/x^{1/\alpha}}$. The result of the current proposition is equivalent to showing that $$\lim_{x\downarrow 0}f^x(1)=C.$$
    
    From the previous bound, we have $f_t^x<(C+\epsilon)t^\alpha+x d_\epsilon$. The functions $f^x(t)$ are thus increasing functions which satisfy $\limsup_{x \downarrow 0}\sup_{t\leq T} f^x(t)\leq (C+\epsilon)T^\alpha$, and 
    $\lim_{\delta\downarrow 0}\limsup_{x \downarrow 0}\sup_{t\leq \delta} f^x(t)=0$. Therefore by \cite[Theorem 12.12.2]{whitt}
    this set of functions is compact under the Skorokhod M1 topology. Let us fix a convergent subsequence, with limit $f_t$. 

    Rearranging the formula from Proposition \ref{Ito}, for any $\varrho>0$
    
    $$\frac{\Gamma(\frac{1}{\alpha})\varrho^{-1/\alpha}}{\alpha C^{1/\alpha}}=\int_0^\infty e^{-\varrho f^{x}(t)-\frac{\varrho^2x^{2-1/\alpha}t}{2}}dt+o(1)_{x\rightarrow 0}.$$
    In the above expression $\Gamma(.)$ is the standard Gamma function, not the operator on cadlag functions used elsewhere in this work.

    Taking $x\rightarrow 0$, and applying Fatou's lemma, we therefore obtain that
    for all $\varrho>0$
    $$\int_0^\infty e^{-\varrho f_t}dt\leq \frac{\Gamma(\frac{1}{\alpha})\varrho^{-1/\alpha}}{\alpha C^{1/\alpha}}.$$

    The right hand side is exactly the value of the integral
    $\int_0^\infty e^{-\varrho Ct^\alpha}dt,$ and so
    $$\int_0^\infty \left(e^{-\varrho f_t}-e^{-\varrho Ct^\alpha}\right)dt\leq 0.$$
    From the upper bound of $Ct^\alpha$ on $f_t$, the integrand is non-negative, hence must be zero, and $f_t=Ct^\alpha$. Since this holds for any convergent subsequence of $f_t^x$ we obtain that $f^x(1)$ indeed converges to $C$ as $x\downarrow0.$
\end{proof}

\begin{remark}
The above result is restricted to the case $\alpha\in [1/2,1]$ as we are unable to show existence of positive initial data leading to a $\Lambda_t=Ct^\alpha$  for $\alpha>1$, hence no comparison can be made. However, for $\alpha>1$, it is simple to obtain a lower bound of $\limsup_{t\rightarrow\infty}\Lambda_t/t^{\alpha}\geq C$ for initial data with an appropriate Laplace transform, by considering the formula from Proposition \ref{Ito}, and using Fatou's lemma.
\end{remark}

\section{The Laplace transform}
\label{itoformula}
We will establish the result of Proposition \ref{Ito} using Ito's formula.

For ease of notation, we assume that $\int_0^\infty g(x)dx=\infty$. Consider a collection $(X_{0^-}^i)_{i\in\mathbb{N}}$ consisting of the points of a Poisson point process of intensity $g$, and which are independent of a collection $(B_t^i)_{i\in\mathbb{N}}$ of independent Brownian motions. For a solution $\Lambda$ to \eqref{alternativetomvr}, with initial data $g$, we define
\begin{align}
   &W_t^i=X_{0^-}^i+B_t^i,\nonumber\\
   &X_t^i=W_t^i-\Lambda_t,\label{aligned:MVR}\\
   &\tau_i=\inf\{t\geq0:X_t^i\leq 0\}.\nonumber
\end{align}

Since $\Lambda$ solves \eqref{alternativetomvr}, we have
\begin{align*}&\Lambda_t=\mathbb{E}\left(\sum_{i=1}^\infty\mathbbm{1}_{\{\tau_i\leq t\}}\right),\\
&V(t,x)=\sum_{i=1}^\infty P(X_{0^-}^i+B_t^i\in \textrm{d}x,\tau_i>t),\\&V(t^-,x)=\sum_{i=1}^\infty\mathbb{P}(X_{0^-}^i+B_t^i\in \textrm{d}x,\tau_i\geq t).\end{align*}
Furthermore, by standard properties of Poisson point processes, $M_t=\Lambda_t-\sum_{i=1}^\infty\mathbbm{1}_{\{\tau_i\leq t\}}$ is a martingale under the natural filtration. 

We first apply Ito's formula to $e^{-\lambda W_t^i-\frac{\lambda^2t}{2}}\mathbbm{1}_{\{\tau_i\leq t\}}$. This yields:
\begin{align*}
    e^{-\lambda W_t^i-\frac{\lambda^2t}{2}}\mathbbm{1}_{\{\tau_i>t\}}-e^{-\lambda X_{0^-}^i}=-\lambda \int_0^t e^{-\lambda W_s^i-\frac{\lambda^2s}{2}}\mathbbm{1}_{\{\tau_i> s\}}dB_s^i-e^{-\lambda W_{\tau_i}^i-\frac{\lambda^2\tau_i}{2}}\mathbbm{1}_{\{\tau_i\leq t\}}.
\end{align*}

Now we take expectations and then sum over $i\in \mathbb{N}$. It is easy to see that each of the terms, when summed over $i$, will be integrable. Thus, we have
\begin{align}
\label{itoterm2}
\mathbb{E}\left(\sum_{i=1}^\infty e^{-\lambda W_t^i-\frac{\lambda^2t}{2}}\mathbbm{1}_{\{\tau_i\geq t\}}\right)-\mathbb{E}\left(\sum_{i=1}^\infty e^{-\lambda X_{0^-}^i}\right)=-\mathbb{E}\left(\sum_{i=1
}^\infty e^{-\lambda W_{\tau_i}^i}\mathbbm{1}_{\{\tau_i\leq t\}}\right).
\end{align}

Write $J$ for the set of discontinuities of $\Lambda_t$.
If $\tau_i\in J^c$, we have $W_{\tau_i}=\Lambda_{\tau_i}$. Replacing the $W_{\tau_i}$ terms by the $\Lambda_{\tau_i}$ terms at all times $\tau_i\notin J$, the term on the right hand side of the above equation can then be rewritten as 
\begin{align}
\label{discsum}-&\mathbb{E}\left(\int_0^te^{-\lambda \Lambda_{s^-}-\frac{\lambda^2s}{2}}d\left(\sum_{i=1}^\infty \mathbbm{1}_{\{\tau_i\leq s\}}\right)\right)\\-&\sum_{s\in J\cap[0,t]}\mathbb{E}\left(\sum_{i=1}^\infty e^{-\frac{\lambda^2s}{2}}\left(e^{-\lambda W_{s}^i}-e^{-\lambda \Lambda_{s^-}}\right)\mathbbm{1}_{\{\tau_i =s\}}\right).\nonumber\end{align}

It is easily seen that the process $\sum_{i=1}^\infty \mathbbm{1}_{\{\tau_i\leq t\}}-\Lambda_t$ is a martingale, by applying the properties of Poisson point processes. Therefore, the first term in \eqref{discsum} can be replaced by the integral
$$\int_0^te^{-\lambda\Lambda_s-\frac{\lambda^2s}{2}}d\Lambda_s.$$
By the definition of $V(t^-,x)$, each term in the sum over $s\in J$ may be rewritten as 
\begin{align*}
    &\mathbb{E}\left(\sum_{i=1}^\infty e^{-\frac{\lambda^2s}{2}}\left(e^{-\lambda W_{s}^i}-e^{-\lambda \Lambda_{s^-}}\right)\mathbbm{1}_{\{\tau_i\geq s\}}\mathbbm{1}_{\{W_s^i\leq \Lambda_s\}}\right)\\
    &=\int_{\Lambda_{s^-}}^{\Lambda_s}e^{-\frac{\lambda^2s}{2}}\left(e^{-\lambda x}-e^{-\lambda \Lambda_{s^-}}\right)V(s^-,x)\textrm{d}x.
\end{align*}

Denoting by $\hat V_t$ the Laplace transform $\int_{0}^\infty V(t,x)e^{-\lambda x}\textrm{d}x$, the equation \eqref{itoterm2} can be rewritten as 
\begin{align*}
    \hat V_te^{-\frac{\lambda^2t}{2}}-\hat V_0&=-\int_0^te^{-\lambda \Lambda_{s^-}-\frac{\lambda^2s}{2}}d\Lambda_s-\sum_{s\leq t}\int_{\Lambda_{s^-}}^{\Lambda_s}e^{-\frac{\lambda^2s}{2}}\left(e^{-\lambda x}-e^{-\lambda \Lambda_{s^-}}\right)V(s^-,x)\textrm{d}x,\\
    &=-\int_0^te^{-\lambda \Lambda_{s^-}-\frac{\lambda^2s}{2}}d\Lambda_s-\sum_{s\leq t}e^{-\frac{\lambda^2s}{2}}\left(\int_{\Lambda_{s^-}}^{\Lambda_s}V(s^-,x)e^{-\lambda x}\textrm{d}x-e^{-\lambda \Lambda_{s^-}}\Delta \Lambda_s\right).
\end{align*}

The first term on the right hand side is given by 
$$\int_0^te^{-\frac{\lambda^2s}{2}}d\left(\frac{e^{-\lambda \Lambda_s}}{\lambda }\right)-\sum_{s\leq t}e^{-\frac{\lambda^2s}{2}}\left(e^{-\lambda \Lambda_{s^-}}\Delta \Lambda_s+\frac{\Delta e^{-\lambda \Lambda_s} }{\lambda}\right).$$
Therefore, applying integration by parts to this term, gives
\begin{align*}
    \hat V_te^{\frac{\lambda^2t}{2}}-\hat V_0&=\frac{\lambda }{2}\int_0^te^{-\lambda \Lambda_s-\frac{\lambda^2s}{2}}\textrm{d}s+\frac{e^{-\lambda \Lambda_t-\frac{\lambda^2t}{2}}-1}{\lambda}-\sum_{s\leq t}e^{-\frac{\lambda^2s}{2}}\left(e^{-\lambda \Lambda_{s^-}}\Delta\Lambda_s+\frac{\Delta e^{-\lambda \Lambda_s}}{\lambda}\right)\\
    &-\sum_{s\leq t}e^{-\frac{\lambda^2s}{2}}\left(\int_{\Lambda_{s^-}}^{\Lambda_s}e^{-\lambda x}V(s^-,x)\textrm{d}x-e^{-\lambda\Lambda_{s^-}}\Delta \Lambda_s\right).\end{align*}
Rearranging this expression, we finally obtain:
\begin{align*}
    &\left(\hat V_t-\frac{e^{-\lambda \Lambda_t}}{\lambda}\right)e^{-\frac{\lambda^2t}{2}}+\left(\frac{1}{\lambda}-\hat V_0\right)\\
    &=\frac{\lambda}{2}\int_0^te^{-\lambda\Lambda_{s^-}-\frac{\lambda^2s}{2}}\textrm{d}s+\sum_{s\leq t}e^{-\frac{\lambda^2s}{2}}\left(-\int_{\Lambda_{s^-}}^{\Lambda_s}V(s^-,x)e^{-\lambda x}\textrm{d}x-\frac{1}{\lambda}\Delta e^{-\lambda\Lambda_s}\right).
\end{align*}
The statement of Proposition~\ref{Ito} then follows.
\section*{Funding}
\thanks{This work is supported by the EPSRC Centre for Doctoral
Training in Mathematics of Random Systems: Analysis, Modelling and Simulation (EP/S023925/1).}
\bibliographystyle{amsplain}
\bibliography{bible}


\end{document}